\documentclass[11pt]{amsart}

\usepackage{amssymb,eucal}

\newtheorem{theorem}{Theorem}[section]

\newtheorem{corollary}[theorem]{Corollary}

\numberwithin{equation}{section}

\theoremstyle{definition}

\newtheorem*{example*}{Example}
\newtheorem{example}[theorem]{Example}

\newtheorem*{remark*}{Remark}

\newcommand\fsl{\mathfrak{sl}}
\newcommand\fgl{\mathfrak{gl}}

\newcommand{\ara}{\stackrel {\varphi_1}{\rightarrow}}
\newcommand{\arb}{\stackrel {\varphi_2}{\rightarrow}}
\newcommand{\arc}{\stackrel {\varphi_i}{\rightarrow}}

\newcommand{\cL}{\mathcal{L}}
\newcommand{\cA}{\mathcal{A}}
\newcommand{\g}{\mathfrak{g}}
\newcommand{\fs}{\mathfrak{s}}

\newcommand{\sfA}{\mathsf{A}}
\newcommand{\sfB}{\mathsf{B}}
\newcommand{\sfC}{\mathsf{C}}
\newcommand{\sfE}{\mathsf{E}}
\newcommand{\sfN}{\mathsf{N}}

\newcommand{\sfZ}{\mathsf{Z}}
\newcommand{\sfI}{\mathsf{I}}
\newcommand{\sfS}{\mathsf{S}}
\newcommand{\FF}{\mathbb{F}}

\newcommand{\VV}{\mathsf{V}}

\DeclareMathOperator{\Der}{\mathsf{Der}}
\DeclareMathOperator{\Hom}{\mathsf{Hom}}
\DeclareMathOperator{\tr}{tr}
\DeclareMathOperator{\ann}{\mathsf{ann}}
\DeclareMathOperator{\ad}{\mathsf{ad}}

\begin{document}

\title[Lie Algebras with Prescribed $\fsl_3$ Decomposition]{Lie Algebras with Prescribed \\
$\fsl_3$ Decomposition}

\author[Georgia Benkart]{Georgia Benkart$^{\star}$}
\thanks{$^{\star}$Part of this work was done during a visit of the first author to the University of Zaragoza, supported by the Spanish Ministerio de Educaci\'on y Ciencia and FEDER (MTM 2007-67884-C04-02).}

\address{Department of Mathematics, University of Wisconsin, Madison, WI 53706, USA}
\email{benkart@math.wisc.edu}

\author[Alberto Elduque]{Alberto Elduque$^{\star\star}$}
\thanks{$^{\star\star}$Supported by the Spanish Ministerios de Educaci\'on y Ciencia and Ciencia e Innovaci\'on and
FEDER (MTM 2007-67884-C04-02 and MTM2010-18370-C04-02) and by the Diputaci\'on General de Arag\'on (Grupo de Investigaci\'on de \'Algebra)}
\address{Departamento de Matem\'aticas e Instituto Universitario de Matem\'aticas y Aplicaciones,
Universidad de Zaragoza, 50009 Zaragoza, Spain}
\email{elduque@unizar.es}




\begin{abstract}  In this work, we consider Lie algebras $\cL$ containing a subalgebra isomorphic to $\fsl_3$ and  such that
$\cL$ decomposes as a module for that $\fsl_3$ subalgebra
into copies of the adjoint module,  the natural 3-dimensional module
and its dual, and the trivial one-dimensional module.   We determine the multiplication in $\cL$ and establish connections with structurable algebras by exploiting symmetry relative to the symmetric group $\sfS_4$.
\end{abstract}

\maketitle


\begin{section} {Introduction}

The Lie algebra $\fgl_{n+k}$ of $(n+k)\times (n+k)$ matrices over a field $\FF$ of characteristic 0 under the commutator product
$[x,y] = xy-yx$, when viewed as a module for the copy of $\fgl_n$ in its northwest corner,
decomposes into  $k$ copies of the natural $n$-dimensional $\fgl_n$-module $\VV = \FF^n$,
$k$ copies of the dual module $\VV^* = \Hom(\VV, \FF)$,  a copy of the Lie algebra
$\fgl_k$ in its southeast corner, and the copy of $\fgl_n$:
\[
\fgl_{n+k} =  \fgl_n \oplus  \VV^{\oplus k} \oplus (\VV^*)^{\oplus k} \oplus \fgl_k.
\]
As a result,  we may write
\[
\fgl_{n+k} \cong  \fgl_n \oplus (\VV \otimes \sfB)  \oplus (\VV^* \otimes \sfC) \oplus \fgl_k,
\]
where $\sfB = \sfC=\FF^k$.  This second expression reflects the decomposition of
$\fgl_{n+k}$ as a module for $\fgl_n \oplus \fgl_k$.   When restricted to $\fsl_n$,   the $\fgl_n$-modules $\VV$ and $\VV^*$ remain irreducible,  while $\fgl_n$ decomposes into a copy of
the adjoint module and a trivial $\fsl_n$-module spanned by the identity matrix:  $\fgl_n = \fsl_n \oplus \FF \sf{I}_n$.
Thus,  we have the $\fsl_n$ decomposition of $\fgl_{n+k}$,
\begin{equation}\label{eq:gdec}
\fgl_{n+k} \cong \fsl_n \oplus  (\VV \otimes \sfB)  \oplus (\VV^* \otimes \sfC) \oplus \big( \fgl_k \oplus \FF \sfI_n\big),
\end{equation}
where $\fgl_k \oplus \FF \sfI_n$ is the sum of the trivial $\fsl_n$-modules in $\fgl_{n+k}$.     Decompositions such
as \eqref{eq:gdec}  also arise in the
study of direct limits of simple Lie algebras and give insight into their structure.

Indeed, suppose we have a chain of homomorphisms,    \begin{equation}\label{dl}
\mathfrak{g}^{(1)}\ara
\mathfrak{g}^{(2)}\arb \ \ldots \ \rightarrow\mathfrak{g}^{(i)}\arc
\mathfrak{g}^{(i+1)}\rightarrow \ \ldots\,,
\end{equation}
where
$\mathfrak{g}^{(i)}=\fsl(\VV^{(i)})$.    Assume that   $\fsl(\VV)$ is  a fixed term in the chain
for some $\VV = \VV^{(j)}$,  and $\dim \VV = n$.   We identify
$\fsl(\VV)$ with $\fsl_n$ by choosing a basis for $\VV$  and assume that   $\VV^{(i)} = \VV^{\oplus k_i} \oplus  \FF^{\oplus z_i}$ as a module for $\fsl_n$
for $i \geq j$.   Then the limit Lie algebra $\cL =
\mathop{\lim}\limits_{\longrightarrow}
\mathfrak{g}^{(i)}$ admits a decomposition relative to $\fsl_n$,
\begin{equation}\label{eq:dec}
\cL \cong  (\fsl_n \otimes \sfA) \oplus (\VV \otimes \sfB) \oplus (\VV^\ast  \otimes \sfC) \oplus \fs,
\end{equation}
where $\fs$ is the sum of the trivial $\fsl_n$-modules (see \cite[Sec.~5]{BB}).  Bahturin and Benkart in \cite[Sec.~4]{BB}
study Lie algebras having such a decomposition and describe the multiplication in $\cL$  and  the possibilities for $\sfA, \sfB, \sfC, \fs$ when
$\dim \VV \geq 4$.    When $\dim \VV = 2$, then
$\VV^*$ is isomorphic to $\VV$ as a module for $\fsl_2 = \fsl(\VV)$.     In this case,  a Lie algebra having a decomposition,
 $\cL =  (\fsl_2  \otimes \sfA) \oplus (\VV \otimes \sfB)  \oplus \fs$  is graded by the root system $\mathsf{BC}_1$,
 and its structure has been described in \cite{BS}.

In this paper,  we investigate the missing case when $\dim \VV = 3$, which presents very distinctive features.
For direct limit Lie algebras of the type considered above, we could,  of course, choose a larger  space $\VV^{(j)}$ having
$\dim \VV^{(j)} \geq 4$  and apply the results of \cite{BB}.  However,
there are many examples of Lie algebras which admit  very interesting decompositions as in \eqref{eq:dec} for $n = 3$.   The exceptional simple Lie algebras provide examples of this phenomenon.

\begin{example}  Each  exceptional simple Lie algebra $\cL$ over an algebraically closed
field of characteristic 0  has an automorphism $\psi$  of order 3 that corresponds to a certain node
in the Dynkin diagram of the associated affine Lie algebra.   The node is marked with a ``3'' in
{\rm \cite[TABLE Aff 1]{K}}.    Removing that node gives the Dynkin diagram of a finite-dimensional semisimple Lie
algebra $\fsl_3 \oplus \mathfrak s$, which is the subalgebra of  fixed points of the automorphism $\psi$.    The Lie algebra
$\mathfrak s$ is the centralizer of $\fsl_3$ in $\cL$; hence, is the sum of trivial $\fsl_3$-modules under the adjoint action.
In this table we display the Lie algebra $\mathfrak s$:
\begin{equation}
\begin{matrix}
\begin{tabular}[t]{|c||c|c|c|c|c|}
\hline
$\cL$ &
$ \mathsf{G}_2 $ & $ \mathsf{F}_4$ & $ \mathsf{E}_6$ &  $\mathsf{E}_7$ &  $\mathsf{E}_8$
\\  \hline
$\mathfrak s$ &  $0$ & $\fsl_3$ &  $\fsl_3 \oplus \fsl_3$  & $\fsl_6$  & $\mathsf{E}_6$ \\
\hline
\end{tabular}
\end{matrix}
\end{equation}
\smallskip

For the Lie algebra $\mathsf{G}_2$ we have the well-known decomposition {\rm (see \cite[Prop.~3] {J})}
\[
\mathsf{G}_2  \cong   \fsl_3 \oplus \VV \oplus \VV^*
\]
relative to $\fsl_3$
(where $\fsl_3$ corresponds to the long roots of $\mathsf{G}_2$ and $\VV = \FF^3$).
This decomposition can be viewed as the decomposition into eigenspaces relative to $\psi$, where
$\VV$ corresponds  to the eigenvalue $\omega$ (a primitive cube root of 1); $\VV^\ast$ to the eigenvalue $\omega^2$;
and $\fsl_3$ to the eigenvalue 1.

For the other exceptional Lie algebras,
\begin{equation}\label{eq:A=F}
\cL \cong   \fsl_3 \oplus  (\VV \otimes \sfB) \oplus (\VV^\ast \otimes \sfC) \oplus \mathfrak s,
\end{equation}
where $\sfB$ and $\sfC$ can be identified with $\mathsf{H}_3(\mathcal C)$,  the algebra of $3 \times 3$ hermitian matrices over a
composition algebra $\mathcal C$ under the product $h \circ h' = {1/2} (hh'+h'h)$.    Thus, elements of $\sfB$
have the form
\[
h =  \left[\begin{array}{ccc}  \alpha & a & b \\  \bar a & \beta & c \\ \bar b & \bar c & \gamma \end{array}\right]
\]
where $\alpha,\beta,\gamma \in \FF$,  $a,b,c \in \mathcal C$, and
 ``$\ \bar{}\ $'' is the standard involution in $\mathcal C$.   The composition algebra $\mathcal C$
is displayed below,
\begin{equation}
\begin{matrix}
\begin{tabular}[t]{|c||c|c|c|c|}
\hline
$\cL$ &
 $ \mathsf{F}_4$ & $ \mathsf{E}_6$ &  $\mathsf{E}_7$ &  $\mathsf{E}_8$
\\  \hline
$\mathcal C$ & $\FF$ &  $\mathsf{K}$  & $\mathsf{Q}$  & $\mathsf{O} $ \\
\hline
\end{tabular}
\end{matrix}
\vrule width 0pt depth 20pt
\end{equation}
where  $\mathsf{K}$ is the algebra $\FF\times\FF$,  $\mathsf{Q}$ the algebra of quaternions,  and $\mathsf{O}$ the algebra of octonions.     The algebra $\mathfrak{s}$ can be identified with
the structure Lie  algebra of $\sfB = \mathsf{H}_3(\mathcal C)$,
$\mathfrak s = \Der(\sfB) \oplus \mathsf{L}_{\sfB_0}$, consisting of the derivations and multiplication maps $\mathsf{L}_h(h') = h \circ h'$ for  $h \in \sfB_0$  (the matrices in $\sfB$ of trace 0).      Here $\VV \otimes \sfB$ is the $\omega$-eigenspace of $\psi$,
 $\VV^\ast \otimes \sfC$ the $\omega^2$-eigenspace,   and $\fsl_3 \oplus \mathfrak{s}$  the $1$-eigenspace.

For example, when $\mathcal C = \mathsf{O}$, it is well known that  $\sfB = \mathsf{H}_3(\mathsf O)$ is the 27-dimensional exceptional simple Jordan algebra,  and its structure algebra $\mathfrak s$ is a simple Lie
algebra of type $\mathsf{E}_6$ (see for example, {\rm \cite[Chap.~IV, Sec.~4]{S})}.
As a module for $\mathsf{E}_6$,  $\sfB$ is irreducible,  and relative to a certain Cartan subalgebra,  it has highest weight the first fundamental weight.
The module $\sfC$ is  an irreducible $\mathsf{E}_6$-module, (the dual module of $\sfB$)  which
has  highest weight the last fundamental weight.
Thus,   $$\mathsf{E}_8 = \fsl_3 \oplus (\VV \otimes \sfB) \oplus (\VV^* \otimes \sfC) \oplus \mathsf{E}_6.$$
Reading right to left, we see the decomposition of $\mathsf{E}_8$ as a module for the subalgebra
of type $\mathsf{E}_6$, and reading left to right, its decomposition as an $\fsl_3$-module.\qed
\end{example}

Recently, Lie algebras with a decomposition \eqref{eq:A=F} have been considered by Faulkner \cite[Lem. 22]{Fau10} in connection with his classification of structurable superalgebras of classical type. (Structurable algebras, which were introduced and studied in \cite{A},  form a certain variety of algebras generalizing associative
algebras with involution and Jordan algebras.)

\bigskip

In this work, we examine Lie algebras $\cL$ such that $\cL$
has a subalgebra $\fsl_3$ and such that $\cL$ admits
a decomposition as in \eqref{eq:dec} into copies of $\fsl_3$, $\VV = \FF^3$, $\VV^\ast$, and trivial
modules relative to the action of $\fsl_3$.
Applying results in \cite{BB} and \cite{BM},  we determine that $\sfA$ is an alternative algebra, $\sfB$ a left $\sfA$-module, and $\sfC$ a right $\sfA$-module, and we describe $\mathfrak s$ and the multiplication in $\cL$.

Using the fact that $\VV$ can be given the structure of a module for the symmetric group $\mathsf{S}_4$, we obtain an action of $\mathsf{S}_4$ by automorphisms on $\cL$.  The elements  $\tau_1 = (1\,2)(3\,4)$ and $\tau_2 = (1\,4)(2 \,3)$ generate a normal subgroup of $\mathsf{S}_4$ which is a Klein 4-subgroup.    Results of Elduque and Okubo \cite{EO} enable us to deduce that  ${\cL}_0 =  \{X \in \cL \mid \tau_1 X = X, \ \tau_2 X = - X\}$ is a structurable algebra under a certain multiplication.
  We identify the structurable algebra  $\cL_0$ with
the space of  $2 \times 2$ matrices
\[
\cA = \left [ \begin{array}{cc}  \sfA & \sfC   \\  \sfB &\sfA \end{array}\right]
\]
under a suitable multiplication.   When $\cL$ is the exceptional Lie algebra $\mathsf{E}_8$,
then  ${ \mathcal A} = \left [ \begin{array}{cc}  \FF & \sfC   \\  \sfB &\FF \end{array}\right]$ where
$\sfB = \sfC = \mathsf{H}_3(\mathsf{O})$.
This is a simple structurable algebra (see \cite[Secs. 8 and 9]{A}).

\end{section}


\begin{section}{Lie algebras with prescribed $\fsl_3$ decomposition}

Let $\cL$ be a Lie algebra over a  field $\FF$ of characteristic $\ne 2,3$ (this assumption on the underlying field will be kept throughout), which contains a subalgebra isomorphic to $\fsl(\VV)$, for a vector space $\VV$ of dimension $3$, so that $\cL$ decomposes, as a module for $\fsl(\VV)$ into a direct sum of copies of the adjoint module, the natural module $\VV$, its dual $\VV^*$,  and the trivial one-dimensional module. Thus, we write as in \eqref{eq:dec}:
\begin{equation}\label{eq:LslVVs}
\cL=\bigl(\fsl(\VV)\otimes \sfA\bigr)\oplus\bigl(\VV\otimes \sfB\bigr)\oplus\bigl(\VV^*\otimes \sfC\bigr)\oplus\fs
\end{equation}
for suitable vector spaces $\sfA,\sfB,\sfC$, and for a Lie subalgebra $\fs$, which is the subalgebra of elements of $\cL$ annihilated by the elements in $\fsl(\VV)$. The vector space $\sfA$ contains a distinguished element $1\in\sfA$ such that $\fsl(\VV)\otimes 1$ is the subalgebra isomorphic to $\fsl(\VV)$ we have started with.

Fix a nonzero linear map $\det:\bigwedge^3 \VV\rightarrow \FF$. This determines another such form $\det: \bigwedge^3\VV^*\rightarrow \FF$ such that $\det(f_1\wedge f_2\wedge f_3)\det(v_1\wedge v_2\wedge v_3)={\boldsymbol \det}\Bigl(f_i(v_j)\Bigr)$ for any $f_1,f_2,f_3\in \VV^*$ and $v_1,v_2,v_3\in \VV$. (The symbol  ``${\boldsymbol \det}$'' denotes the usual determinant.)

This allows us to identify $\bigwedge^2 \VV$ with $\VV^*$: $u_1\wedge u_2\leftrightarrow \det(u_1\wedge u_2\wedge \underline{\ \ })$ and, in the same vein, $\bigwedge^2 \VV^*$ with  $\VV$.

\smallskip

The invariance of the bracket in $\cL$ relative to the subalgebra $\fsl(\VV)$ gives equations as in \cite[(19)]{BB}:

\begin{equation}\label{eq:bracket}
\begin{split}
[x\otimes a,y\otimes a']&=[x,y]\otimes \frac{1}{2}a\circ a'\, +\, x\circ y\otimes \frac{1}{2} [a,a']+(x\vert y)D_{a,a'},\\
[x\otimes a,u\otimes b]&=xu\otimes ab,\\
[v^*\otimes c,x\otimes a]&=v^*x\otimes ca,\\
[u\otimes b,v^*\otimes c]&=\bigl(uv^*-\frac{1}{3}(v^*u)\sfI_3\bigr)\otimes T(b,c)+\frac{1}{3}(v^*u)D_{b,c},\\
[u_1\otimes b_1,u_2\otimes b_2]&= (u_1\wedge u_2) \otimes (b_1\times b_2),\\
[v_1^*\otimes c_1,v_2^*\otimes c_2]&= (v_1^*\wedge v_2^*) \otimes (c_1\times c_2),\\
[d,x\otimes a]&=x\otimes da,\\
[d,u\otimes b]&=u\otimes db,\\
[d,v^*\otimes c]&=v^*\otimes dc,
\end{split}
\end{equation}
for any $x,y\in\fsl(\VV)$, $u,u_1,u_2\in \VV$, $v^*,v_1^*,v_2^*\in\VV^*$, $d\in\fs$, $a,a'\in\sfA$, $b,b_1,b_2\in\sfB$ and $c,c_1,c_2\in\sfC$, and for bilinear maps:
\begin{equation}\label{eq:bilinearmaps}
\begin{split}
&\sfA\times\sfA\rightarrow\sfA: (a,a')\mapsto a\circ a'\  \ \text{commutative},\\
&\sfA\times\sfA\rightarrow\sfA: (a,a')\mapsto [a,a']\  \ \text{anticommutative},\\
&\sfA\times\sfA\rightarrow\fs, (a,a')\mapsto D_{a,a'}\ \  \text{skew-symmetric},\\
&\sfA\times\sfB\rightarrow\sfB: (a,b)\mapsto ab,\\
&\sfC\times \sfA\rightarrow\sfC: (c,a)\mapsto ca,\\
&\sfB\times\sfC\rightarrow\sfA: (b,c)\mapsto T(b,c),\\
&\sfB\times\sfC\rightarrow\fs: (b,c)\mapsto D_{b,c},\\
&\sfB\times\sfB\rightarrow\sfC: (b_1,b_2)\mapsto b_1\times b_2\ \ \text{symmetric},\\
&\sfC\times\sfC\rightarrow\sfB: (c_1,c_2)\mapsto c_1\times c_2\ \ \text{symmetric},
\end{split}
\end{equation}
and representations $\fs\rightarrow \fgl(\sfA),\fgl(\sfB),\fgl(\sfC)$, whose action is denoted by $da$, $db$, and $dc$ for $d\in \fs$ and $a\in \sfA$, $b\in\sfB$ and $c\in\sfC$;  where, as in \cite[(17)]{BB},
\begin{equation}\label{eq:slnotation}
\begin{split}
x\circ y&=xy+yx-\frac{2}{3}\tr(xy)\sfI_3,\\
(x\vert y)&=\frac{1}{3}\tr(xy),
\end{split}
\end{equation} for $x,y\in\fsl(\VV)$, and $\sfI_3$ denotes the identity map.
The difference with \cite[(19)]{BB} lies in the appearance of the symmetric maps $b_1\times b_2$ and $c_1\times c_2$ when $\VV$ has dimension 3. This slight difference has a huge impact.

The distinguished element $1\in\sfA$ satisfies $1\circ a=a$, $[1,a]=0$, $D_{1,a}=0$, $1b=b$, $c1=c$,  and $d1=0$ for any $a\in \sfA$, $b\in\sfB$, $c\in\sfC$ and $d\in\fs$.

\smallskip

\begin{theorem}\label{th:Lie}
Let $\cL$ be a vector space as in \eqref{eq:LslVVs} and define an anticommutative bracket in $\cL$ by \eqref{eq:bracket} for bilinear maps as in \eqref{eq:bilinearmaps}. Then $\cL$ is a Lie algebra if and only if the following conditions are satisfied:
\begin{enumerate}
\item[(0)] $\fs$ is a Lie subalgebra of $\cL$, $\sfA$, $\sfB$, $\sfC$ are modules for $\fs$ relative to the given actions, and the bilinear maps in \eqref{eq:bilinearmaps} are $\fs$-invariant.

\item[(1)] $\sfA$ is an alternative algebra relative to the multiplication
    \[
    aa'=\frac{1}{2}a\circ a'+\frac{1}{2}[a,a'],
    \]
    and the map $A\times A\rightarrow A: (a,a')\mapsto D_{a,a'}$ satisfies the conditions
    \[
    \begin{split}
    &\sum_{\text{\rm cyclic}}D_{a_1,a_2a_3}=0,\\ &D_{a_1,a_2}a_3=[[a_1,a_2],a_3]+3\bigl((a_1a_3)a_2-a_1(a_3a_2)\bigr),
    \end{split}
    \]
    for any $a_1,a_2,a_3\in \sfA$.

\item[(2)] For any $a_1,a_2\in\sfA$, $b\in \sfB$ and $c\in\sfC$,
    \[
    \begin{split}
    &a_1(a_2b)=(a_1a_2)b,\\
    &(ca_1)a_2=c(a_1a_2),
    \end{split}
    \]
    so that $\sfB$ (respectively $\sfC$) is a left associative module (resp. right associative module) for $\sfA$, and
    \[
    \begin{split}
    &D_{a_1,a_2}b=[a_1,a_2]b,\\
    &D_{a_1,a_2}c=c[a_2,a_1].
    \end{split}
    \]

\item[(3)] For any $a\in\sfA$, $b\in\sfB$ and $c\in\sfC$,
    \[
    \begin{split}
    &aT(b,c)=T(ab,c),\quad T(b,c)a=T(b,ca),\\
    &D_{a,T(b,c)}=D_{ab,c}-D_{b,ca},\\
    &D_{b,c}a=[T(b,c),a].
    \end{split}
    \]

\item[(4)] For any $a\in \sfA$, $b_1,b_2\in\sfB$ and $c_1,c_2\in\sfC$,
    \[
    \begin{split}
    &(b_1\times b_2)a=(ab_1)\times b_2=b_1\times(ab_2),\\
    &a(c_1\times c_2)=(c_1a)\times c_2=c_1\times (c_2a).
    \end{split}
    \]

\item[(5)] $D_{b,b\times b}=0$ for any $b\in\sfB$  and $D_{c\times c,c}=0$ for any $c\in\sfC$. In addition,  the trilinear maps $\sfB\times\sfB\times\sfB\rightarrow \sfB: (b_1,b_2,b_3)\mapsto T(b_1,b_2\times b_3)$  and $\sfC\times\sfC\times\sfC\rightarrow \sfC: (c_1,c_2,c_3)\mapsto T(c_1\times c_2,c_3)$ are symmetric.

\item[(6)] For any $b,b_1,b_2\in \sfB$ and $c,c_1,c_2\in\sfC$,
    \[
    \begin{split}
    (b_1\times b_2)\times c&
     =-\frac{1}{3}T(b_1,c)b_2+\frac{1}{3}D_{b_1,c}b_2-T(b_2,c)b_1,\\
    b\times (c_1\times c_2)&
    =-c_2T(b,c_1)-\frac{1}{3}c_1T(b,c_2)-\frac{1}{3}D_{b,c_2}c_1.
    \end{split}
    \]
\end{enumerate}
\end{theorem}
\begin{proof}
If $\cL$ is a Lie algebra under the bracket defined in \eqref{eq:bracket}, then it is clear that $\fs$ is a Lie subalgebra and all the conditions in item (0) are satisfied. Moreover, $\bigl(\fsl(\VV)\otimes \sfA)\oplus\fs$ is a Lie subalgebra, and the arguments in \cite[Sec. 3]{BM}
show that $\sfA$ is an alternative algebra,  and the conditions in item (1) are satisfied.

The arguments in Propositions 4.3 and 4.4,  and Equations (25) and (27) in \cite{BB},  work here and give the conditions in item (2). (Note that there is a minus sign missing in \cite[(27)]{BB}.)

Now, equations (30)--(33) in \cite{BB} establish the identitites in (3). The Jacobi identity applied to elements $x\otimes a$, $u_1\otimes b_1$ and $u_2\otimes b_2$, for $x\in\fsl(\VV)$, $u_1,u_2\in \VV$, $a\in\sfA$ and $b_1,b_2\in\sfB$, give the first equation in item (4), the second one being similar.

The Jacobi identity for elements $u_i\otimes b_i$, $i=1,2,3$, for $u_i\in \VV$ and $b_i\in\sfB$  gives:
\[
\sum_{\text{cyclic}}D_{b_1,b_2\times b_3}=0,\quad T(b_1,b_2\times b_3)=T(b_2,b_3\times b_1),
\]
which, in view of the symmetry of the bilinear map $b_1\times b_2$, proves half of the assertions in item (5);   the other half being implied by the Jacobi identity for elements $v_i^*\otimes c_i$, $i=1,2,3$, for $v_i^*\in \VV^*$ and $c_i\in\sfC$.

Finally, for elements $u_1,u_2\in \VV$, $v^*\in \VV^*$, $b_1,b_2\in\sfB$, $c\in\sfC$:
\[
\begin{split}
&[[u_1\otimes b_1,u_2\otimes b_2],v^*\otimes c]
  =[(u_1\wedge u_2)\otimes (b_1\times b_2),v^*\otimes c]\\
&\quad =(u_1\wedge u_2)\wedge v^* \otimes (b_1\times b_2)\times c
 =\bigl((v^*u_1)u_2-(v^*u_2)u_1\bigr)\otimes (b_1\times b_2)\times c,
\end{split}
\]
while
\[
\begin{split}
&[[u_1\otimes b_1,v^*\otimes c],u_2\otimes b_2]\\
&\quad =[\bigl(u_1v^*-\frac{1}{3}(v^*u_1)\sfI_3\bigr)\otimes T(b_1,c)
   +\frac{1}{3}v^*u_1D_{b_1,c},u_2\otimes b_2]\\
&\quad =\bigl((v^*u_2)u_1-\frac{1}{3}(v^*u_1)u_2\bigr)\otimes T(b_1,c)b_2
  +\frac{1}{3}(v^*u_1)u_2\otimes D_{b_1,c}b_2,\\[6pt]
&[u_1\otimes b_1,[u_2\otimes b_2,v^*\otimes c]]\\
&\quad [u_1\otimes b_1,\bigl(u_2v^*-\frac{1}{3}(v^*u_2)\sfI_3\bigr)\otimes T(b_2,c)
   +\frac{1}{3}v^*u_2D_{b_2,c}]\\
&\quad =-\bigl((v^*u_1)u_2-\frac{1}{3}(v^*u_2)u_1\bigr)\otimes T(b_2,c)b_1
  -\frac{1}{3}(v^*u_2)u_1\otimes D_{b_2,c}b_1.
\end{split}
\]
Hence the Jacobi identity here is equivalent to the first condition in item (6);  the second condition
can be proven in a similar way.

\smallskip

The converse follows from straightforward computations.
\end{proof}

\smallskip

Given an alternative algebra $\sfA$, the ideal $\sfE(\sfA)$ generated by the associators $(a_1,a_2,a_3)=(a_1a_2)a_3-a_1(a_2a_3)$ is $\sfE(\sfA)=(\sfA,\sfA,\sfA)+(\sfA,\sfA,\sfA)\sfA=(\sfA,\sfA,\sfA)+\sfA(\sfA,\sfA,\sfA)$. The associative nucleus of $\sfA$ is $\sfN(\sfA):=\{a\in\sfA \mid (a,\sfA,\sfA)=0\}$, while the center is $\sfZ(\sfA)=\{a\in \sfN(\sfA) \mid  aa'=a'a, \  \ \forall \  a'\in\sfA\}$.

\begin{corollary}\label{co:Lie}
Let $\cL$ be a Lie algebra which contains a subalgebra isomorphic to $\fsl(\VV)$ for a vector space $\VV$ of dimension $3$, so that $\cL$ decomposes, as a module for $\fsl(\VV)$,  as in \eqref{eq:LslVVs}. Then, with the notation used so far, the alternative algebra $\sfA$ is unital (the distinguished element $1$ being its unit element), with $1$ acting as the identity on both $\sfB$ and $\sfC$, and the following conditions hold:
\begin{itemize}
\item $\sfE(\sfA)\sfB=0=\sfC \sfE(\sfA)$, so that $\sfB$ (respectively $\sfC$) is a left (resp. right) module for the associative algebra $\sfA/\sfE(\sfA)$.
\item $T(\sfB,\sfC)$ is an ideal of $\sfA$ contained in its associative nucleus $\sfN(\sfA)$, and $T(\sfB,\sfB\times \sfB)$ and $T(\sfC\times \sfC,\sfC)$ are ideals of $\sfA$ contained in $\sfZ(\sfA)$.
\item For any $b,b_1,b_2\in\sfB$ and any $c,c_1,c_2\in\sfC$, the following conditions hold:
    \[
    \begin{split}
    &D_{b_1,c}b_2-D_{b_2,c}b_1=2\bigl(T(b_2,c)b_1-T(b_1,c)b_2\bigr),\\
    &D_{b,c_2}c_1-D_{b,c_1}c_2=2\bigl(c_1T(b,c_2)-c_2T(b,c_1)\bigr).
    \end{split}
    \]
\item If the Lie algebra $\cL$ is simple, then either the algebra $\sfA$ is associative, or else $\sfA=\sfE(\sfA)$ and $\sfB=\sfC=0$. Moreover, if $\sfB\ne 0$, then $\sfC$ coincides with $\sfB\times\sfB$, and $\sfA$ coincides with $T(\sfB,\sfB\times\sfB)$,  and $\sfA$  is a commutative and associative algebra.
\end{itemize}
\end{corollary}
\begin{proof}
For any $a_1,a_2,a_3\in \sfA$ and $b\in \sfB$,
\[
\begin{split}
(a_1,a_2,a_3)b&=(a_1a_2)(a_3b)-a_1((a_2a_3)b)\\
&=a_1(a_2(a_3b))-a_1(a_2(a_3b))=0,
\end{split}
\]
because of Theorem \ref{th:Lie}, item (2). Also, this result shows that $\ann_\sfA(\sfB)=\{a\in \sfA\mid a\sfB=0\}$ is an ideal of $\sfA$. Hence, $\sfE(\sfA)\sfB=0$, as $\sfE(\sfA)$ is the ideal generated by $(\sfA,\sfA,\sfA)$.
In a similar manner,  one proves $\sfC \sfE(\sfA)=0$.

For any $b\in\sfB$ and $c\in\sfC$, $T(b,c)$ is an element of $\sfA$, and $\ad_{T(b,c)}:a\mapsto [T(b,c),a]=D_{b,c}a$ is a derivation of $\sfA$ by the previous theorem.  Since $\sfA$ is alternative, this shows that $T(b,c)$ is in the associative nucleus $\sfN(\sfA)$. Now for $b_1,b_2,b_3\in \sfB$, $aT(b_1,b_2\times b_3)=T(ab_1,b_2\times b_3)=T(b_2,b_3\times(ab_1))=T(b_2,(b_3\times b_1)a)=T(b_2,b_3\times b_1)a=T(b_1,b_2\times b_3)a$, which proves that $T(\sfB,\sfB\times\sfB)$ is an ideal of $\sfA$ contained in the center $\sfZ(\sfA)$.  By similar arguments,  $T(\sfC\times\sfC,\sfC)$ is shown to be contained in $\sfZ(\sfA)$ too.

For any $b_1,b_2\in\sfB$ and $c\in\sfC$, the previous theorem gives,
\[
(b_1\times b_2)\times c
     =-\frac{1}{3}T(b_1,c)b_2+\frac{1}{3}D_{b_1,c}b_2-T(b_2,c)b_1.
\]
We permute $b_1$ and $b_2$ and use the fact that $\times$ is symmetric to get
\[
D_{b_1,c}b_2-D_{b_2,c}b_1=2\bigl(T(b_2,c)b_1-T(b_1,c)b_2\bigr).
\]
With the same arguments we prove
\[
D_{b,c_2}c_1-D_{b,c_1}c_2=2\bigl(c_1T(b,c_2)-c_2T(b,c_1)\bigr)
\]
for any $b\in\sfB$, and $c_1,c_2\in\sfC$.

Finally, since the ideal $\sfE(\sfA)$ of the alternative algebra $\sfA$ is invariant under derivations, the subspace $(\fsl(\VV)\otimes \sfE(\sfA)\bigr)\oplus D_{\sfE(\sfA),\sfA}$ is an ideal of the Lie algebra $\cL$. In particular, if $\cL$ is simple, then either $\sfA=\sfE(\sfA)$ and $\sfB=\sfC=0$,   or
$\sfE(\sfA) = 0$ and $\sfA$ is associative. Moreover, if $\sfB$ is nonzero, the ideal of $\cL$ generated by $\VV\otimes \sfB$ is $\bigl(\fsl(\VV)\otimes T(\sfB,\sfC)\bigr)\oplus\bigl(\VV\otimes \sfB\bigr)\oplus\bigl(\VV^*\otimes (\sfB\times \sfB)\bigr)\oplus D_{\sfB,\sfC}$.
Hence if $\cL$ is simple,  we obtain $\sfC=\sfB\times\sfB$ and $\sfA=T(\sfB,\sfC)=T(\sfB,\sfB\times\sfB)$, which is commutative and associative.
\end{proof}

\end{section}


\begin{section}{Structurable algebras}

This section is devoted to establishing a relationship between the Lie algebras with prescribed $\fsl_3$ decomposition considered above with a class of structurable algebras. This will be done by exploiting the action of a subgroup of the group of automorphisms of the Lie algebra isomorphic to the symmetric group $\sfS_4$.

\begin{theorem}\label{th:structurable}
Let $\cL$ be a Lie algebra which contains a subalgebra isomorphic to $\fsl(\VV)$ for a vector space $\VV$ of dimension $3$, so that $\cL$ decomposes, as a module for $\fsl(\VV)$,  as in \eqref{eq:LslVVs}. Then, with the notation used so far, the vector space
\begin{equation}\label{eq:calA}
\cA=\begin{pmatrix} \sfA&\sfC\\ \sfB& \sfA\end{pmatrix},
\end{equation}
with the multiplication
\begin{equation}\label{eq:structurable}
\begin{pmatrix} a_1& c\\ b&a_2\end{pmatrix}\cdot\begin{pmatrix} a_1'& c'\\ b'&a_2'\end{pmatrix} =\begin{pmatrix} a_1a_1'-T(b',c)& c'a_1+ca_2'+b\times b'\\ a_1'b+a_2b'+c\times c'&a_2'a_2-T(b,c')\end{pmatrix} \end{equation}
\noindent and  the involution
\begin{equation}\label{eq:structurable2}
\overline{\begin{pmatrix} a_1& c\\ b&a_2\end{pmatrix}} =\begin{pmatrix} a_2& c\\ b&a_1\end{pmatrix}
\end{equation}
is a structurable algebra.
\end{theorem}
\begin{proof}
Take a basis $\{e_1,e_2,e_3\}$ of $\VV$ with $\det(e_1\wedge e_2\wedge e_3)=1$ and its dual basis $\{e_1^*,e_2^*,e_3^*\}$ in $\VV^*.$

The symmetric group $\sfS_4$ acts on $\VV$ as follows \cite[(7.1)]{EO}:
\[
\begin{split}
\tau_1=(12)(34):&\quad e_1\mapsto e_1,\ e_2\mapsto -e_2,\ e_3\mapsto -e_3,\\
\tau_2=(23)(14):&\quad e_1\mapsto -e_1,\ e_2\mapsto e_2,\ e_3\mapsto -e_3,\\
\varphi=(123):&\quad e_1\mapsto e_2\mapsto e_3\mapsto e_1,\\
\tau=(12):&\quad e_1\mapsto -e_1,\ e_2\mapsto -e_3,\ e_3\mapsto -e_2.
\end{split}
\]
(Thus $\VV$ is the tensor product of the sign module and the standard irreducible 3-dimensional module for $\sfS_4$, and in this way,  $\sfS_4$ embeds in the special linear group $\mathsf{SL}(\VV)$.)

The inner product given by $(e_i\vert e_j)=\delta_{ij}$ for any  $i,j \in \{1,2,3\}$ is invariant under the action of $\sfS_4$, so $\VV$ is selfdual as  an $\sfS_4$-module, and the action of $\sfS_4$ on $\VV^*$ (where $\sigma v^*=v^*\sigma^{-1}$) is given by the ``same formulas'':
\[
\begin{split}
\tau_1=(12)(34):&\quad e_1^*\mapsto e_1^*,\ e_2^*\mapsto -e_2^*,\ e_3^*\mapsto -e_3^*,\\
\tau_2=(23)(14):&\quad e_1^*\mapsto -e_1^*,\ e_2^*\mapsto e_2^*,\ e_3^*\mapsto -e_3^*,\\
\varphi=(123):&\quad e_1^*\mapsto e_2^*\mapsto e_3^*\mapsto e_1^*,\\
\tau=(12):&\quad e_1^*\mapsto -e_1^*,\ e_2^*\mapsto -e_3^*,\ e_3^*\mapsto -e_2^*.
\end{split}
\]

Since $\sfS_4$ acts by elements in $\mathsf{SL}(\VV)$, this action of $\sfS_4$ on $\VV$ and on $\VV^*$ extends to an action  by automorphisms on the whole algebra $\cL$. Then the subspace
\[
\cL_0=\{ X\in \cL \mid  \tau_1X=X,\ \tau_2 X=-X\}
\]
becomes a structurable algebra \cite[Thm. 7.5]{EO} with involution and multiplication given by the following formulas,
\[
\begin{split}
\bar  X&=-\tau X,\\
X\cdot Y&= -\tau\bigl([\varphi X,\varphi^2 Y]\bigr),
\end{split}
\]
for any $X,Y\in \cL_0$.

But we easily deduce that
\[
\cL_0= (e_2e_3^*\otimes \sfA)\oplus(e_3e_2^*\otimes \sfA)\oplus (e_1\otimes \sfB)\oplus (e_1^*\otimes \sfC).
\]

Identifying $\cL_0$ with the  $2 \times 2$ matrices $\cA=\begin{pmatrix} \sfA&\sfC\\ \sfB& \sfA\end{pmatrix}$ by means of
\[
-e_2e_3^*\otimes a_1+ e_3e_2^*\otimes a_2 +e_1\otimes b + e_1^*\otimes c\leftrightarrow \begin{pmatrix} a_1& c\\ b&a_2\end{pmatrix},
\]
we determine that the structurable product and the involution become
\[
\begin{split}
\begin{pmatrix} a_1& c\\ b&a_2\end{pmatrix}\cdot\begin{pmatrix} a_1'& c'\\ b'&a_2'\end{pmatrix}&=\begin{pmatrix} a_1a_1'-T(b',c)& c'a_1+ca_2'+b\times b'\\ a_1'b+a_2b'+c\times c'&a_2'a_2-T(b,c')\end{pmatrix}\\
\overline{\begin{pmatrix} a_1& c\\ b&a_2\end{pmatrix}}&=\begin{pmatrix} a_2& c\\ b&a_1\end{pmatrix},
\end{split}
\]
as required.
\end{proof}

\smallskip

Items (5) and (6) of Theorem \ref{th:Lie} show that for any $b\in \sfB$ and $c\in \sfC$,
\begin{equation}\label{eq:bbbb}
\begin{split}
(b\times b)\times(b\times b)&=-\frac{4}{3}T(b,b\times b)b,\\
(c\times c)\times(c\times c)&=-\frac{4}{3}cT(c\times c,c).
\end{split}
\end{equation}
Also, using Theorem \ref{th:Lie} and Corollary \ref{co:Lie} we compute that
\[
\begin{split}
(c\times (b\times b))\times b&=-\frac{4}{3}(T(b,c)b)\times b+\frac{1}{3}(D_{b,c}b)\times b\\
 &=-\frac{4}{3}(b\times b)T(b,c)+\frac{1}{6}D_{b,c}(b\times b)\quad\text{(as the product}\\[-2pt]
 &\qquad\qquad\qquad \text{$b_1\times b_2$ from $\sfB\times\sfB$ into $\sfC$ is $\fs$-invariant)}\\
 &=-\frac{4}{3}(b\times b)T(b,c)+\frac{1}{3}\bigl((b\times b)T(b,c)-cT(b,b\times b)\bigr)\\
 &=-(b\times b)T(b,c)-\frac{1}{3}cT(b,b\times b),
\end{split}
\]
and an analogous result with the roles of $b$ and $c$ interchanged. So we conclude that the equations
\begin{equation}\label{eq:bbbc}
\begin{split}
(c\times(b\times b))\times b&=-(b\times b)T(b,c)-\frac{1}{3}cT(b,b\times b),\\
(b\times (c\times c))\times c&=-T(b,c)(c\times c)-\frac{1}{3}T(c\times c,c)b,
\end{split}
\end{equation}
hold for any $b\in\sfB$ and $c\in\sfC$.

\smallskip

Equations \eqref{eq:bbbb} and \eqref{eq:bbbc} are precisely the ones that appear in \cite[Ex. 6.4]{AF}, and are needed to ensure that the algebra defined there, which coincides with our $\cA$,  but with the added restrictions of $\sfA$ being commutative and associative, is structurable. (Note that the bilinear form $T(.,.)$ considered in \cite[Ex. 6.4]{AF} equals our $-T(.,.)$.)

However some of the previous arguments show that, if our structurable algebra $\cA$ is simple and $\sfA\ne 0$, then $\sfA$ is simple, and since $T(\sfB,\sfB\times \sfB)$ and $T(\sfC\times \sfC,\sfC)$ are ideals of $\sfA$ contained in the center $\sfZ(\sfA)$, either $\sfA$ is commutative and associative, or else $T(\sfB,\sfB\times \sfB)=0=T(\sfC\times \sfC,\sfC)$. But in this case,  the subspace $\begin{pmatrix} 0&\sfB\times \sfB\\ \sfC\times \sfC&0\end{pmatrix}$ becomes an ideal, so if $\cA$ is simple either $\sfA$ is commutative and associative, or else $\sfB\times \sfB=0=\sfC\times \sfC$.

Therefore, when considering simple algebras, we are dealing exactly with the situation considered by Allison and Faulkner in \cite{AF}.

\medskip

Theorem \ref{th:structurable} shows that the restrictions on the bilinear maps involved are sufficient to ensure that the algebra $\cA$ in \eqref{eq:calA}, with multiplication
\eqref{eq:structurable} and involution \eqref{eq:structurable2}  is a structurable algebra.

A natural question to ask is whether these conditions are also necessary.  More precisely, does any structurable algebra of the form $\cA$ as in \eqref{eq:calA} with multiplication \eqref{eq:structurable} and involution \eqref{eq:structurable2}, constructed from a unital alternative algebra $\sfA$, left and right unital ``associative'' modules $\sfB$ and $\sfC$, and  bilinear maps $T(b,c)$, $b_1\times b_2$, and $c_1\times c_2$, coordinatize a Lie algebra $\cL$ with a subalgebra isomorphic to $\fsl(\VV)$ for a vector space $\VV$ of dimension $3$ and  with decomposition as in \eqref{eq:LslVVs}?    (We do not impose any further conditions on these bilinear maps besides requiring that the resulting algebra $\cA$ be structurable.)

Our last result answers this question in the affirmative.

\begin{theorem}\label{th:converse}
Let $\sfA$ be a unital alternative algebra;  let $\sfB$ (respectively $\sfC$) be a left (respectively right) unital associative module for $\sfA$;  and let $\sfB\times\sfC\rightarrow \sfA: (b,c)\mapsto T(b,c)$, $\sfB\times\sfB\rightarrow \sfC: (b_1,b_2)\mapsto b_1\times b_2$, and  $\sfC\times\sfC\rightarrow \sfB: (c_1,c_2)\mapsto c_1\times c_2$ be bilinear maps which make the vector space $\cA$ in \eqref{eq:calA} with the multiplication \eqref{eq:structurable} and involution in \eqref{eq:structurable2} into a structurable algebra. Then there is a Lie algebra $\cL$ containing a subalgebra isomorphic to $\fsl(\VV)$, for a vector space $\VV$ of dimension $3$, such that $\cL$ decomposes as in \eqref{eq:LslVVs} for a suitable vector space $\fs$,  such that the Lie bracket on $\cL$ is given by \eqref{eq:bracket} for some bilinear maps
$\sfA\times\sfA\rightarrow\fs, (a,a')\mapsto D_{a,a'}$,
$\sfB\times\sfC\rightarrow\fs: (b,c)\mapsto D_{b,c}$, $\fs\times\sfA\rightarrow \sfA: (d,a)\mapsto da$, $\fs\times\sfB\rightarrow \sfB: (d,b)\mapsto db$ and $\fs\times\sfC\rightarrow \sfC: (d,c)\mapsto dc$.
\end{theorem}
\begin{proof}
Consider the Lie algebra $\cL=\mathcal{K}(\cA,-,\gamma,\mathcal{V})$ in \cite[Sec.~4]{AF} attached to the structurable algebra $(\cA,-)$, the triple $\gamma=(1,1,1)$, and the Lie subalgebra $\mathcal{V}=\mathcal{T}_I$. This Lie algebra $\cL$, which coincides with the Lie algebra $\g(\cA,\cdot,-)$ in \cite[Ex. 3.1]{EO1}, is the direct sum
\[
\cL=\mathcal{T}_I\oplus\cA[12]\oplus\cA[23]\oplus\cA[31],
\]
where $\mathcal{T}_I$ is the span of the triples $T = (T_1,T_2,T_3)$  with
\begin{equation}\label{eq:TI}
\begin{split}
T_i&=L_{\bar x}L_y-L_{\bar y}L_x,\\
T_j&=R_{\bar x}R_y-R_{\bar y}R_x,\\
T_k&=R_{\bar xy-\bar yx}+L_yL_{\bar x}-L_xL_{\bar y},
\end{split}
\end{equation}
for $x,y\in\cA$ and $(i,j,k)$ a cyclic permutation of $(1,2,3)$. Here $L_xy=xy=R_yx$. The subspace $\mathcal{T}_I$ is a Lie algebra with componentwise bracket, and the Lie bracket in $\cL$ is given by extending the bracket in $\mathcal{T}_I$ by setting $x[ij]=-\bar x[ji]$ for any $x\in\cA$ and
\begin{equation}\label{eq:KATI}
\begin{split}
[x[ij],y[jk]]&=-[x[jk],y[ij]]=(xy)[ik],\\
[T,x[ij]]&=-[x[ij],T]=T_k(x)[ij],\\
[x[ij],y[ij]]&=T,
\end{split}
\end{equation}
for $x,y\in\cA$, where $(i,j,k)$ is a cyclic permutation of $(1,2,3)$, and $T = (T_1, T_2, T_3)$ is as in \eqref{eq:TI}. Theorems 4.1 and 5.5 in \cite{AF} show that $\cL$ is indeed a Lie algebra. Since we are assuming
that  the characteristic of the field is $\ne 2,3$,  Corollary 3.5 of \cite{AF} shows that $\mathcal{T}_I=\{(D,D,D)\mid D\in \Der(\cA,\cdot,-)\}\oplus \{(L_{s_2}-R_{s_3},L_{s_3}-R_{s_1},L_{s_1}-R_{s_2})\mid  s_i\in \cA,\ \bar s_i=-s_i,\ s_1+s_2+s_3=0\}$.   Here $\Der(\cA,\cdot,-)$ is the space of derivations relative
to the product ``$\cdot$'' which commute with the involution ``$-$''.

For any $a\in\sfA$, consider the linear span $\fsl_3[a]$ of the elements
$\left(\begin{smallmatrix}a&0\\ 0&0\end{smallmatrix}\right)[ij]$
for $i\ne j$ and the triples $(L_{\alpha_2 s}-R_{\alpha_3 s},L_{\alpha_3 s}-R_{\alpha_1 s},L_{\alpha_1 s}-R_{\alpha_2 s})$ for $\alpha_i\in \FF$ with $\alpha_1+\alpha_2+\alpha_3=0$ and $s=\left(\begin{smallmatrix}a&\ 0\\ 0&-a\end{smallmatrix}\right)$. (Note that $\left(\begin{smallmatrix}a&0\\ 0&0\end{smallmatrix}\right)[ij]=-\left(\begin{smallmatrix}0&0\\ 0&a\end{smallmatrix}\right)[ji]$.)

Also for any $b\in \sfB$, consider the linear span $\VV[b]$ of the elements
$\left(\begin{smallmatrix}0&0\\ b&0\end{smallmatrix}\right)[ij]$, and for any $c\in\sfC$ the linear span $\VV^*[c]$ of the elements
$\left(\begin{smallmatrix}0&c\\ 0&0\end{smallmatrix}\right)[ij]$.

Straightforward computations using \eqref{eq:KATI} imply that
\begin{itemize}
\item $\fsl_3[1]$ is a Lie subalgebra of $\cL$ isomorphic to $\fsl(\VV)$ ($\dim\VV=3$),
\item $\fsl_3[a]$ is an adjoint module for $\fsl_3[1]$ for any $a\in\sfA$,
\item $\VV[b]$ is the natural module for $\fsl_3[1]$ for any $b\in\sfB$,
\item $\VV^*[c]$ is the dual module for $\fsl_3[1]$ for any $c\in\sfC$, and
\item $\fs=\{(D,D,D) \mid D\in\Der(\cA,\cdot,-)\}$ is a Lie subalgebra which commutes with $\fsl_3[1]$.
\end{itemize}Actually, if we fix a basis $\{e_1,e_2,e_3\}$ of $\VV$ as before with $\det(e_1\wedge e_2\wedge e_3)=1$ and the dual basis $\{e_1^*,e_2^*,e_3^*\}$ in $\VV^*$, we may identify
$\fsl_3[a]$ with $\fsl(\VV)\otimes a$ for $a\in\sfA$ by means of
\[
\begin{split}
e_ie_j^*\otimes a&\leftrightarrow \left(\begin{smallmatrix}a&0\\ 0&0\end{smallmatrix}\right)[ij], \qquad \hbox {\rm for} \ \ i \neq j\\
\sum_{i=1}^3\alpha_i e_ie_i^*\otimes a&\leftrightarrow
(L_{\alpha_2 s}-R_{\alpha_3 s},L_{\alpha_3 s}-R_{\alpha_1 s},L_{\alpha_1 s}-R_{\alpha_2 s}),
\end{split}
\]
for $\alpha_i\in\FF$ with $\alpha_1+\alpha_2+\alpha_3=0$ and $s=\left(\begin{smallmatrix}a&\ 0\\ 0&-a\end{smallmatrix}\right)$.
Also for any $b\in\sfB$ and $c\in\sfC$,  we may identify $\VV[b]$ with $\VV\otimes b$ and $\VV^*[c]$ with $\VV^*\otimes c$ via
\[
e_i\otimes b\leftrightarrow \left(\begin{smallmatrix}0&0\\ b&0\end{smallmatrix}\right)[jk],\qquad
e_i^*\otimes c\leftrightarrow \left(\begin{smallmatrix}0&c\\ 0&0\end{smallmatrix}\right)[jk]
\]
where $(i,j,k)$ is a cyclic permutation of $(1,2,3)$.

In this way we recover the decomposition in \eqref{eq:dec} with bracket as in \eqref{eq:bracket} for suitable maps $D_{.,.}$, as required.
\end{proof}

\end{section}


\end{document}